\documentclass[11pt]{article}

\usepackage[a4paper,margin=1in]{geometry}
\usepackage{amsmath,amssymb,amsfonts,amsthm}
\usepackage{mathrsfs}
\usepackage{mathtools}
\usepackage{enumitem}
\usepackage{hyperref}
\usepackage{color}
\usepackage{fancyhdr}

\hypersetup{
    colorlinks=true,
    linkcolor=blue,
    citecolor=blue,
    urlcolor=blue
}

\newtheorem{theorem}{Theorem}[section]
\newtheorem{proposition}[theorem]{Proposition}
\newtheorem{lemma}[theorem]{Lemma}
\newtheorem{corollary}[theorem]{Corollary}

\theoremstyle{definition}
\newtheorem{definition}[theorem]{Definition}
\newtheorem{convention}[theorem]{Convention}
\newtheorem{setting}[theorem]{Setting}

\theoremstyle{remark}

\newcommand{\gr}{\operatorname{gr}}
\newcommand{\md}{\operatorname{mod}}
\newcommand{\Gproj}{\operatorname{Gproj}}

\newcommand{\Hom}{\operatorname{Hom}}
\newcommand{\Ext}{\operatorname{Ext}}
\newcommand{\rad}{\operatorname{rad}}
\newcommand{\pd}{\operatorname{pd}}
\newcommand{\id}{\operatorname{id}}

\newcommand{\Coker}{\operatorname{Coker}}
\newcommand{\Tr}{\operatorname{Tr}}
\newcommand{\grl}{\operatorname{gr.l.}}
\newcommand{\lb}{\underline{\delta}}
\newcommand{\ub}{\overline{\delta}}
\newcommand{\Z}{\mathbb Z}
\newcommand{\cN}{\mathcal N}
\newcommand{\cI}{\mathcal I}
\newcommand{\cC}{\mathcal C}
\newcommand{\GG}{G_{\mathcal G}}

\title{A Gordon--Green-type result for Gorenstein-projective modules}

\author{
A. M. Bouhada and M. Laaraj
}

\date{}

\begin{document}

\maketitle
\pagestyle{fancy}


\begin{abstract}
Let \(\Lambda\) be a \(\Z\)-graded Artin algebra which is
Iwanaga--Gorenstein. We prove that \(\Lambda\) is CM-finite if and
only if the category of gradable Gorenstein-projective modules has
finite type, and if and only if the lengths of the indecomposable
gradable Gorenstein-projective modules are bounded. Moreover, if
\(\Lambda\) is CM-finite, then every Gorenstein-projective
\(\Lambda\)-module is gradable. These results provide
Gorenstein-projective analogues of the corresponding theorems of
Gordon--Green for graded Artin algebras of finite representation
type.
\end{abstract}
\medskip
\noindent\textbf{Keywords:}
Gorenstein-projective modules; graded Artin algebras; gradable modules;
Cohen--Macaulay finite type; Iwanaga--Gorenstein algebras.

\medskip

\noindent\textbf{2020 Mathematics Subject Classification.}
Primary 16G50; Secondary 16G60, 16G70, 16E65.

\begin{center}
\section{Introduction}
\end{center}

Let \(\Lambda\) be a graded Artin algebra. Gordon and Green
\cite{GG2} studied \(\Lambda\) through the full subcategory
\(\md_\infty(\Lambda)\) of \(\md\Lambda\) consisting of gradable
modules. Among their main results, they proved that \(\Lambda\) has
finite representation type if and only if
\(\md_\infty(\Lambda)\) has finite representation type
\cite[Theorem~4.5]{GG2}, and that finite representation type implies
that every \(\Lambda\)-module is gradable
\cite[Theorem~4.3]{GG2}.

Finite Cohen--Macaulay type and Gorenstein-projective modules have
been studied from several complementary points of view. Chen
\cite{Ch} obtained a Gorenstein-projective analogue of Auslander's
theorem on algebras of finite representation type. Beligiannis
\cite{Be05,Be} developed a systematic study of virtually Gorenstein
algebras and algebras of finite Cohen--Macaulay type. Further
results on CM-finite Gorenstein algebras and constructions of
Gorenstein-projective modules were obtained by Li and Zhang
\cite{LiZ1,LiZ2}, while Kong and Zhang \cite{KZ} investigated
relative Auslander algebras associated with CM-finite algebras. In
the graded setting, Minamoto and Yamaura \cite{MY} studied finitely
graded Iwanaga--Gorenstein algebras and the stable categories of
their graded Cohen--Macaulay modules.

Explicit descriptions are available for important classes of
algebras. In particular, Chen, Shen and Zhou \cite{CSZ} classified
the indecomposable non-projective Gorenstein-projective modules over
a monomial algebra in terms of perfect paths. More recently, Honma and Usui \cite{HU} described both the stable category of Gorenstein-projective modules and the stable category of graded Gorenstein-projective modules over a monomial algebra, together with the relation between them via the orbit category with respect to the degree-shift functor. These results
illustrate the close interaction between grading and
Gorenstein-projective representation theory.

The question considered here is different. We study how
CM-finiteness interacts with gradability and ask whether the
finiteness phenomena established by Gordon and Green for ordinary
modules continue to hold inside the Gorenstein-projective category.
Recall that \(\Lambda\) is \emph{CM-finite} if
\(\Gproj\Lambda\) has only finitely many indecomposable objects up
to isomorphism.

The following main theorem shows that the Gordon--Green finiteness
phenomenon extends to the Gorenstein-projective setting.

\begin{theorem}[Main Theorem]\label{thm:intro}
Let \(\Lambda\) be a graded Artin algebra which is
Iwanaga--Gorenstein. The following conditions are equivalent:
\begin{enumerate}[label=\textup{(\arabic*)}]
\item \(\Lambda\) is CM-finite.
\item The category \(\Gproj_\infty(\Lambda)\) of gradable
Gorenstein-projective modules has finite type.
\item The lengths of the indecomposable gradable
Gorenstein-projective modules are uniformly bounded.
\item \(\GG<\infty\) and \(\Gproj_d(\Lambda)\) has finite type for
every \(d\geq0\).
\end{enumerate}

Moreover, if \(\Lambda\) is CM-finite, then every
Gorenstein-projective \(\Lambda\)-module is gradable. Consequently,
the ungraded stable category of Gorenstein-projective modules can be
realized as the orbit category of the graded stable category of
graded Gorenstein-projective modules with respect to the degree-shift
functor \(\langle 1\rangle\). In particular, one has
\(
\underline{\Gproj}\Lambda
\simeq
\underline{\Gproj}^{\Z}\Lambda/\langle 1\rangle\
\).
\end{theorem}

The notation appearing in \textup{(4)} is introduced in
Section~\ref{sec:proof}. When \(\Lambda\) is self-injective, every
finitely generated \(\Lambda\)-module is Gorenstein-projective, and
the theorem recovers the corresponding finiteness phenomena of
Gordon and Green.

\medskip
\noindent\textbf{Strategy of the proof.}
We combine right almost split morphisms in $\md\Lambda$ with
Gorenstein-projective approximations. In the graded setting, these
approximations are chosen so that gradability is preserved, yielding
right almost split morphisms in $\Gproj\Lambda$ whose domains are
gradable whenever their targets are. Iterating the indecomposable
components of these morphisms produces chains in $\Gproj\Lambda$,
along which gradability propagates from the last term to all
preceding terms. Under the bounded-length hypothesis in
\textup{(3)}, an expansion argument, together with the Harada--Sai
lemma, shows that every indecomposable Gorenstein-projective
module is, up to isomorphism, the initial term of a chain of
uniformly bounded length ending at an indecomposable projective
module. This yields both CM-finiteness and the gradability of
all Gorenstein-projective modules.

Our approach is inspired by Auslander's use of the Harada--Sai
lemma to show that a component of the Auslander--Reiten quiver
consisting of modules of bounded length contains all
indecomposable modules \cite[Chapter~VI]{ARS}; a relative version
appears in \cite[Remark~3.2(ii)]{Be}. Unlike the approach of
Gordon and Green \cite{GG2}, our proof uses neither Auslander's
theorem on components nor the connectedness of $\Lambda$.
The essential additional ingredient in our setting is the
propagation of gradability along the chosen chains, which
allows the length bound to be imposed only on gradable modules.

The same argument also provides a short proof of the results
of Gordon and Green \cite[Theorems~4.3 and 4.5]{GG2}.
Indeed, working in $\md\Lambda$ instead of $\Gproj\Lambda$,
neither Gorenstein-projective approximations nor the
Gorenstein hypothesis are needed. Moreover, the chains may
be taken to end at simple modules rather than projective ones,
since every nonzero module admits a nonzero morphism onto a
simple module, and there are only finitely many simple modules,
all of which are gradable.


\begin{center}
\section{Proof of the Main Theorem}\label{sec:proof}
\end{center}

Before giving the proof, we recall the notation and auxiliary
results that will be used throughout this section.


\subsection{Graded Artin algebras}

Throughout, \(\Lambda\) is a graded Artin algebra and all modules are
finitely generated left modules. We write \(\ell(M)\) for the
length of a module \(M\) and \(\Omega M\) for the kernel of a
projective cover.

Let \(\Lambda=\bigoplus_{i\in\Z}\Lambda_i\) be a graded Artin
algebra. We denote by \(\gr\Lambda\) the category of finitely
generated graded \(\Lambda\)-modules with degree-preserving
morphisms, and by \(F\colon\gr\Lambda\to\md\Lambda\) the forgetful
functor. For \(X\in\gr\Lambda\) and \(i\in\Z\), denote by
\(X\langle i\rangle\) the grading shift of \(X\), defined by
\((X\langle i\rangle)_n=X_{n-i}\).

A module \(M\in\md\Lambda\) is \emph{gradable} if \(M\cong FX\) for
some \(X\in\gr\Lambda\). We denote by \(\md_\infty(\Lambda)\) the
full subcategory of gradable modules.

For \(0\neq X\in\gr\Lambda\), let \(\lb X\) and \(\ub X\) denote,
respectively, the least and largest integers \(n\) for which
\(X_n\neq0\), and put \(\grl X=\ub X-\lb X+1\). For \(d\geq0\), let
\(\gr^d(\Lambda)\) be the full subcategory of \(\gr\Lambda\)
consisting of \(0\) and those \(X\neq0\) satisfying
\(\lb X\geq0\) and \(\ub X\leq d\). Put
\(\md_d(\Lambda)=F(\gr^d(\Lambda))\), with closure under isomorphism
understood.

We shall use the following standard results of Gordon--Green
\cite{GG1,GG2}.

\begin{proposition}\label{prop:GG}
Let \(X,Y\in\gr\Lambda\).
\begin{enumerate}[label=\textup{(GG\arabic*)}]
\item\label{GG:exact}
The functor \(F\) is exact and faithful and reflects exactness.

\item\label{GG:indec}
The module \(X\) is indecomposable if and only if \(FX\) is
indecomposable. Moreover, direct summands and finite direct sums of
gradable modules are gradable.

\item\label{GG:shift}
If \(X\) and \(Y\) are indecomposable, then \(FX\cong FY\) if and
only if \(X\cong Y\langle i\rangle\) for some \(i\in\Z\).

\item\label{GG:hom}
Every morphism \(FX\to FY\) is a finite sum of homogeneous
morphisms. A degree-\(e\) morphism \(FX\to FY\) is the image under
\(F\) of a degree-preserving morphism
\(X\to Y\langle -e\rangle\).

\item\label{GG:rad}
We have \(F(\rad X)=\rad FX\).

\item\label{GG:cover}
The functor \(F\) sends graded projective covers to projective
covers and graded injective envelopes to injective envelopes. In
particular, \(F(\Omega^nX)\cong\Omega^n(FX)\) for all \(n\geq0\).

\item\label{GG:DTr}
There is a graded Auslander--Reiten translate such that
\(F(D\Tr X)\cong D\Tr(FX)\).

\item\label{GG:AR}
For every indecomposable non-injective \(X\in\gr\Lambda\), there is
an almost split sequence starting at \(X\), and \(F\) sends almost
split sequences in \(\gr\Lambda\) to almost split sequences in
\(\md\Lambda\).

\item\label{GG:length}
If \(X\) is indecomposable and \(\ell(FX)\leq K\), then
\(\grl X\leq K\grl\Lambda\).
\end{enumerate}
\end{proposition}

Since \(\Lambda\) is an Artin algebra, only finitely many
homogeneous components \(\Lambda_i\) are nonzero. Hence
\(\grl\Lambda<\infty\).


\subsection{Right almost split maps and Gorenstein-projective modules}

Let \(\cC\) be a full subcategory of \(\md\Lambda\). A morphism
\(\vartheta\colon Y\to B\) with \(Y,B\in\cC\) is
\emph{right almost split in \(\cC\)} if it is not a split
epimorphism and every morphism \(W\to B\) in \(\cC\) which is not a
split epimorphism factors through \(\vartheta\).

Every indecomposable \(B\in\md\Lambda\) admits a right almost split
map. If \(B\) is non-projective, this is the right-hand morphism of
the almost split sequence ending at \(B\); if \(B\) is projective,
the inclusion \(\rad B\hookrightarrow B\) is right almost split
\cite[Chapter~V]{ARS}.

We shall use the following form of the Harada--Sai lemma.

\begin{lemma}[Harada--Sai {\cite{HS}}]\label{lem:HS}
Let \(b\geq1\). A composite of \(2^b-1\) non-isomorphisms between
indecomposable modules of length at most \(b\) is zero.
\end{lemma}

We write \(\Gproj\Lambda\) for the category of finitely generated
Gorenstein-projective \(\Lambda\)-modules.

Assume that \(\Lambda\) is Iwanaga--Gorenstein and let
\(d=\id{}_\Lambda\Lambda\). We shall use the following standard
facts \cite{Bu,EJ}:
\begin{enumerate}[label=\textup{(G\arabic*)}]
\item\label{G:char}
A module \(M\) is Gorenstein-projective if and only if
\(\Ext^i_\Lambda(M,\Lambda)=0\) for all \(i\geq1\).

\item\label{G:syz}
For every \(M\in\md\Lambda\), one has
\(\Omega^dM\in\Gproj\Lambda\).

\item\label{G:closed}
The category \(\Gproj\Lambda\) contains the projective modules and
is closed under extensions, finite direct sums and direct summands.

\item\label{G:embed}
Every \(G\in\Gproj\Lambda\) embeds into a projective module.

\item\label{G:orth}
If \(G\in\Gproj\Lambda\) and \(\pd L<\infty\), then
\(\Ext^i_\Lambda(G,L)=0\) for all \(i\geq1\).
\end{enumerate}

\begin{definition}
Let \(\cC\) be a full subcategory of \(\md\Lambda\). A morphism
\(\pi\colon C\to M\) with \(C\in\cC\) is a
\emph{right \(\cC\)-approximation} of \(M\) if every morphism
\(C'\to M\) with \(C'\in\cC\) factors through \(\pi\).
\end{definition}

We shall use the Auslander--Buchweitz approximation: every
\(M\in\md\Lambda\) admits an exact sequence
\(0\to K\to G\to M\to0\) with \(G\in\Gproj\Lambda\) and
\(\pd K<\infty\). By \ref{G:orth}, the epimorphism \(G\to M\) is a
right \(\Gproj\Lambda\)-approximation \cite{AB}.

We now introduce the graded Gorenstein-projective category.

\begin{definition}\label{def:gradedGproj}
We denote by \(\Gproj^{\Z}\Lambda\) the full subcategory of
\(\gr\Lambda\) consisting of all \(X\) such that
\(\Ext^i_{\gr\Lambda}(X,\Lambda\langle j\rangle)=0\) for every \(i\geq1\)
and every \(j\in\Z\).
\end{definition}

\begin{proposition}\label{prop:gradedGproj}
For \(X\in\gr\Lambda\), one has \(X\in\Gproj^{\Z}\Lambda\) if and
only if \(FX\in\Gproj\Lambda\). Consequently,
\(F(\Gproj^{\Z}\Lambda)=\Gproj_\infty(\Lambda)\).
\end{proposition}

\begin{proof}
Let \(P^\bullet\to X\) be a graded projective resolution. By
\ref{GG:hom}, for every \(n\), the module
\(\Hom_\Lambda(FP^n,\Lambda)\) is naturally isomorphic to
\(\bigoplus_{j\in\Z}
\Hom_{\gr\Lambda}(P^n,\Lambda\langle j\rangle)\).
The differentials preserve this homogeneous decomposition.
Consequently, for every \(i\geq0\),
\(\Ext^i_\Lambda(FX,\Lambda)\) is naturally isomorphic to
\(\bigoplus_{j\in\Z}
\Ext^i_{\gr\Lambda}(X,\Lambda\langle j\rangle)\).
The first assertion follows from \ref{G:char}, and the second from
the definition of gradability.
\end{proof}

By \ref{GG:indec} and \ref{GG:shift}, \(F\) induces a bijection
between the isomorphism classes of indecomposable objects of
\(\Gproj^{\Z}\Lambda\) modulo shifts and the isomorphism classes of
indecomposable objects of \(\Gproj_\infty(\Lambda)\).


\subsection{Graded right almost split maps}

\begin{lemma}\label{lem:gradedras}
Let \(C\in\gr\Lambda\) be indecomposable. Then there exists a
morphism \(\vartheta'\colon Y'\to C\) in \(\gr\Lambda\) such that
\(F\vartheta'\) is right almost split in \(\md\Lambda\).
\end{lemma}

\begin{proof}
Suppose first that \(FC\) is projective. By \ref{GG:rad}, the image
under \(F\) of \(\rad C\hookrightarrow C\) is
\(\rad FC\hookrightarrow FC\), which is right almost split.

Suppose now that \(FC\) is non-projective and put \(X=D\Tr C\). By
\ref{GG:DTr}, \(FX\cong D\Tr(FC)=\tau(FC)\). Hence \(X\) is
indecomposable by \ref{GG:indec}. Moreover, \(X\) is not injective:
otherwise \(FX\) would be injective by \ref{GG:cover}, whereas
\(\tau(FC)\) is the first term of the almost split sequence ending
at \(FC\).

By \ref{GG:AR}, there is an almost split sequence
\(0\to X\to Y\xrightarrow{\alpha}Z\to0\) in \(\gr\Lambda\). Its
image under \(F\) is almost split in \(\md\Lambda\). Since
\(FX\cong\tau(FC)\), we have \(FZ\cong FC\). Both \(Z\) and \(C\)
are indecomposable, so \ref{GG:shift} gives
\(Z\cong C\langle i\rangle\) for some \(i\in\Z\). Shifting the sequence and
composing its right-hand morphism with an isomorphism
\(Z\langle -i\rangle\to C\) gives the required morphism.
\end{proof}


\subsection{Graded Gorenstein-projective approximations}

\begin{lemma}\label{lem:leftapprox}
Let \(X\in\Gproj^{\Z}\Lambda\). Then there is an exact sequence
\(0\to X\xrightarrow{\iota}Q\to X''\to0\) in \(\gr\Lambda\) such
that \(Q\) is projective in \(\gr\Lambda\) and
\(X''\in\Gproj^{\Z}\Lambda\).
\end{lemma}

\begin{proof}
Set \(H=\Hom_\Lambda(FX,\Lambda)\). Since \(\Lambda\) is Artinian,
\(H\) is finitely generated as a right \(\Lambda\)-module. By
\ref{GG:hom}, after replacing a finite generating set by its
homogeneous components, we may choose homogeneous generators
\(\varphi_1,\dots,\varphi_m\). Write
\(\varphi_k=F\varphi'_k\), where
\(\varphi'_k\colon X\to\Lambda\langle -e_k\rangle\).

Put \(Q=\bigoplus_{k=1}^m\Lambda\langle -e_k\rangle\) and
\(\iota=(\varphi'_k)_k\). By construction, every morphism
\(FX\to\Lambda\), and hence every morphism from \(FX\) to a
projective module, factors through \(F\iota\).

By \ref{G:embed}, there is a monomorphism \(u\colon FX\to P\) with
\(P\) projective. Since \(u\) factors through \(F\iota\), the
morphism \(F\iota\) is a monomorphism. Hence \(\iota\) is a
monomorphism by \ref{GG:exact}.

Let \(X''=\Coker\iota\). Applying \(\Hom_\Lambda(-,\Lambda)\) to
\(0\to FX\to FQ\to FX''\to0\), the surjectivity of
\(\Hom_\Lambda(FQ,\Lambda)\to\Hom_\Lambda(FX,\Lambda)\) gives
\(\Ext^1_\Lambda(FX'',\Lambda)=0\). For \(i\geq2\), one has
\(\Ext^i_\Lambda(FX'',\Lambda)\cong
\Ext^{i-1}_\Lambda(FX,\Lambda)=0\).
Thus \(FX''\in\Gproj\Lambda\) by \ref{G:char}. Proposition
\ref{prop:gradedGproj} gives \(X''\in\Gproj^{\Z}\Lambda\).
\end{proof}

\begin{proposition}[Graded Auslander--Buchweitz approximation]
\label{prop:approx}
For every \(M\in\gr\Lambda\), there is an exact sequence
\(0\to K\to G\xrightarrow{\pi}M\to0\) in \(\gr\Lambda\) such that
\(G\in\Gproj^{\Z}\Lambda\) and \(\pd FK<\infty\). Moreover, \(F\pi\)
is a right \(\Gproj\Lambda\)-approximation of \(FM\).
\end{proposition}

\begin{proof}
We prove by induction on \(n\geq0\) that the assertion holds
whenever \(\Omega^nM\in\Gproj^{\Z}\Lambda\). This is sufficient by
\ref{GG:cover}, \ref{G:syz}, and Proposition
\ref{prop:gradedGproj}.

For \(n=0\), take \(G=M\) and \(K=0\). Let \(n\geq1\), and take a
graded projective cover \(0\to\Omega M\to P\to M\to0\). By
induction there is an exact sequence
\(0\to K'\to G'\to\Omega M\to0\) with
\(G'\in\Gproj^{\Z}\Lambda\) and \(\pd FK'<\infty\).
Lemma~\ref{lem:leftapprox} gives an exact sequence
\(0\to G'\to Q\to G''\to0\), where \(Q\) is graded projective and
\(G''\in\Gproj^{\Z}\Lambda\).

Forming the pushout of \(G'\to\Omega M\) and \(G'\to Q\) gives
exact sequences \(0\to K'\to Q\to N\to0\) and
\(0\to\Omega M\to N\to G''\to0\). Thus \(\pd FN<\infty\).
Pushing out \(\Omega M\to N\) along \(\Omega M\to P\) gives
\(0\to N\to L\to M\to0\) and \(0\to P\to L\to G''\to0\).
Hence \(FL\in\Gproj\Lambda\) by \ref{G:closed}, so
\(L\in\Gproj^{\Z}\Lambda\) by Proposition
\ref{prop:gradedGproj}.

Finally, if \(W\in\Gproj\Lambda\), then
\(\Ext^1_\Lambda(W,FN)=0\) by \ref{G:orth}. Hence every morphism
\(W\to FM\) lifts through \(FL\to FM\). Thus \(FL\to FM\) is a
right \(\Gproj\Lambda\)-approximation.
\end{proof}


\subsection{Right almost split morphisms in \(\Gproj\Lambda\)}

\begin{proposition}\label{prop:rasGproj}
Every indecomposable \(B\in\Gproj\Lambda\) admits a right almost split
morphism in \(\Gproj\Lambda\). If \(B\) is gradable, then such a
morphism may be chosen with gradable domain.
\end{proposition}

\begin{proof}
Choose a right almost split morphism
\(\vartheta\colon Y\to B\) in \(\md\Lambda\) and a right
\(\Gproj\Lambda\)-approximation \(\pi\colon E\to Y\). Put
\(\theta=\vartheta\pi\). If \(\theta\) were a split epimorphism,
then \(\vartheta\) would be a split epimorphism.

Let \(v\colon W\to B\), with \(W\in\Gproj\Lambda\), be a morphism
which is not a split epimorphism. Since \(\vartheta\) is right almost
split, \(v=\vartheta u\) for some \(u\colon W\to Y\). Since \(\pi\)
is a right \(\Gproj\Lambda\)-approximation, \(u=\pi u'\) for some
\(u'\colon W\to E\). Hence \(v=\theta u'\).

Suppose now that \(B\) is gradable. Choose an indecomposable
\(C\in\gr\Lambda\) and an isomorphism \(w\colon FC\to B\).
Lemma~\ref{lem:gradedras} gives
\(\vartheta'\colon Y'\to C\) such that \(F\vartheta'\) is right
almost split in \(\md\Lambda\). By Proposition~\ref{prop:approx},
there is \(\pi'\colon E'\to Y'\) with
\(E'\in\Gproj^{\Z}\Lambda\) and \(F\pi'\) a right
\(\Gproj\Lambda\)-approximation. The preceding argument shows that
\(wF(\vartheta'\pi')\colon FE'\to B\) is right almost split in \(\Gproj\Lambda\). Its domain is gradable.
\end{proof}

\begin{convention}\label{conv}
Fix a set \(\cI\) of representatives of the isomorphism classes of
indecomposable objects of \(\Gproj\Lambda\). For every \(B\in\cI\),
choose a right almost split morphism \(\theta_B\colon E_B\to B\) in
\(\Gproj\Lambda\), with
\(E_B\) gradable whenever \(B\) is gradable, and fix a decomposition
\(E_B=\bigoplus_{\ell=1}^{s_B}E_{B,\ell}\) into indecomposable
modules. Write \(\theta_{B,\ell}\colon E_{B,\ell}\to B\) for the
corresponding components.

For each \(1\leq\ell\leq s_B\), let
\(\widehat E_{B,\ell}\in\cI\) be the representative of the
isomorphism class of \(E_{B,\ell}\), and fix an isomorphism
\(v_{B,\ell}\colon\widehat E_{B,\ell}\to E_{B,\ell}\). Put
\(\widehat\theta_{B,\ell}
=\theta_{B,\ell}v_{B,\ell}\colon\widehat E_{B,\ell}\to B\).
\end{convention}

Thus all modules which occur in the chains considered below belong to the
fixed set \(\cI\).

\begin{lemma}\label{lem:noniso}
For every \(B\in\cI\) and every \(1\leq\ell\leq s_B\), the
morphism
\(\widehat\theta_{B,\ell}\colon\widehat E_{B,\ell}\to B\)
is not an isomorphism.
\end{lemma}

\begin{proof}
If \(\widehat\theta_{B,\ell}\) were an isomorphism, then
\(\theta_{B,\ell}\) would also be an isomorphism because
\(v_{B,\ell}\) is an isomorphism. This would make \(\theta_B\) a
split epimorphism, a contradiction.
\end{proof}

\begin{definition}\label{def:componentchain}
A \emph{chain of components of the chosen right almost split morphisms} of
length \(m\geq0\) is a sequence
\begin{equation*}
Z_0\xrightarrow{f_1}Z_1\xrightarrow{f_2}\cdots
\xrightarrow{f_m}Z_m
\end{equation*}
with \(Z_i\in\cI\) for all \(i\), such that for every
\(1\leq i\leq m\) there is some \(\ell\) with
\(Z_{i-1}=\widehat E_{Z_i,\ell}\) and
\(f_i=\widehat\theta_{Z_i,\ell}\).
\end{definition}

\begin{lemma}\label{lem:gradablechain}
If such a chain \(Z_0\to\cdots\to Z_m\) ends at a gradable
module \(Z_m\), then every \(Z_i\) is gradable.
\end{lemma}

\begin{proof}
Suppose \(Z_i\) is gradable. By Convention~\ref{conv}, \(E_{Z_i}\)
is gradable. Hence each direct summand \(E_{Z_i,\ell}\) is gradable.
Since \(\widehat E_{Z_i,\ell}\cong E_{Z_i,\ell}\), the
representative \(\widehat E_{Z_i,\ell}\) is also gradable.
Therefore \(Z_{i-1}\) is gradable, and the result follows by
descending induction.
\end{proof}


\subsection{The expansion lemma}
The following lemma provides a finite expansion of a
non-isomorphism between indecomposable Gorenstein-projective
modules in terms of compositions of components of the chosen
right almost split morphisms. The expansion will allow us to
apply the Harada--Sai lemma to sufficiently long compositions
while retaining control over the shorter ones.
\begin{lemma}\label{lem:expansion}
Let \(M,P\in\cI\) and let \(f\colon M\to P\) be a non-isomorphism.
For every \(n\geq1\), there is a finite family
\((c_k,h_k)_{k\in K}\) such that
\(f=\sum_{k\in K}c_kh_k\) and, for every \(k\in K\):
\begin{enumerate}[label=\textup{(\roman*)}]
\item \(c_k\) is the composite of such a chain
\(Z^{(k)}_0\to\cdots\to Z^{(k)}_{m_k}=P\), with
\(1\leq m_k\leq n\);
\item \(h_k\colon M\to Z^{(k)}_0\), and \(h_k\) is an isomorphism
whenever \(m_k<n\).
\end{enumerate}
\end{lemma}

\begin{proof}
Since \(M\) and \(P\) are indecomposable, \(f\) is not a split
epimorphism. We proceed by induction on \(n\).

For \(n=1\), the right almost split property of
\(\theta_P\colon E_P\to P\) yields a factorization
\(f=\theta_Ph\). Writing \(h=(h_\ell)_\ell\) and setting
\(\widehat h_\ell=v_{P,\ell}^{-1}h_\ell\), we obtain
\[
f=\sum_\ell
\widehat\theta_{P,\ell}\widehat h_\ell,
\]
which gives the required expansion.

Suppose the assertion holds for \(n\), and write
\(f=\sum_{k\in K}c_kh_k\) accordingly. Terms with
\(m_k<n\), or with \(m_k=n\) and \(h_k\) an isomorphism,
are retained. For each remaining term, \(m_k=n\) and
\(h_k\colon M\to Z^{(k)}_0\) is a non-isomorphism.
By the right almost split property of
\(\theta_{Z^{(k)}_0}\), we may write
\[
h_k=\sum_\ell
\widehat\theta_{Z^{(k)}_0,\ell}h_{k,\ell},
\qquad
h_{k,\ell}\colon M\to
\widehat E_{Z^{(k)}_0,\ell}.
\]
Consequently,
\[
c_kh_k=\sum_\ell
c_k\widehat\theta_{Z^{(k)}_0,\ell}h_{k,\ell},
\]
where each new composite is a chain of length \(n+1\).
The resulting finite expansion satisfies both conditions
for \(n+1\), completing the induction.
\end{proof}


\subsection{Proof of the Main Theorem}\label{sec:main}

Let
\(\Gproj_\infty(\Lambda)=
\Gproj\Lambda\cap\md_\infty(\Lambda)\) and
\(\Gproj_d(\Lambda)=
\Gproj\Lambda\cap\md_d(\Lambda)\) for \(d\geq0\). Put
\(\GG=\sup\{\grl X:
X\in\Gproj^{\Z}\Lambda\text{ indecomposable}\}\).

We first prove the key bounded-length statement.

\begin{proposition}\label{prop:key}
Suppose there is \(b\geq1\) such that \(\ell(M)\leq b\) for every
indecomposable gradable \(M\in\Gproj\Lambda\). Then every
indecomposable non-projective \(M\in\Gproj\Lambda\) is isomorphic to
the initial term of a chain of components of the chosen right almost split
morphisms of length at most
\(2^b-2\) ending at an indecomposable projective module. In
particular, every Gorenstein-projective \(\Lambda\)-module is
gradable, and every indecomposable Gorenstein-projective module has
length at most \(b\).
\end{proposition}

\begin{proof}
Let \(M\) be an indecomposable non-projective object of
\(\Gproj\Lambda\), and let \(M_0\in\cI\) be its representative. It
is enough to prove the assertion for \(M_0\).

By \ref{G:embed}, \(M_0\) embeds into a projective module.
Decomposing the latter into indecomposable projectives, some
component gives a nonzero morphism \(M_0\to P'\), where \(P'\) is
indecomposable projective. Let \(P\in\cI\) be the representative of
\(P'\). Since projectivity is invariant under isomorphism, \(P\) is
projective. Composing with an isomorphism \(P'\to P\), we obtain a
nonzero morphism \(f\colon M_0\to P\). Since \(M_0\) is
non-projective, \(f\) is not an isomorphism.

Set \(n=2^b-1\) and apply Lemma~\ref{lem:expansion}. Thus
\(f=\sum_{k\in K}c_kh_k\), where \(c_k\) is the composite of such a chain
\(Z^{(k)}_0\to\cdots\to Z^{(k)}_{m_k}=P\), with
\(1\leq m_k\leq n\).

The module \(P\) is gradable. Hence all modules in these chains are
gradable by Lemma~\ref{lem:gradablechain}. By hypothesis they have
length at most \(b\). Moreover, every arrow in such a chain
is a non-isomorphism by Lemma~\ref{lem:noniso}. Therefore
\(c_k=0\) whenever \(m_k=n=2^b-1\), by the Harada--Sai lemma.

If \(b=1\), then \(n=1\), and hence \(m_k=1\) for every \(k\).
Therefore \(c_k=0\) for all \(k\), so
\(f=\sum_{k\in K}c_kh_k=0\), contradicting \(f\neq0\). Thus, if an
indecomposable non-projective Gorenstein-projective module exists,
then \(b\geq2\).

Assume therefore that \(b\geq2\). Since \(f\neq0\) and
\(c_k=0\) whenever \(m_k=n\), there is some \(k\in K\) such that
\(m_k<n\) and \(c_kh_k\neq0\). By Lemma~\ref{lem:expansion},
\(h_k\colon M_0\to Z^{(k)}_0\) is an isomorphism. Hence \(M_0\),
and therefore \(M\), is isomorphic to the initial term of such a chain
of length at most \(2^b-2\) ending at the
indecomposable projective module \(P\). Lemma~\ref{lem:gradablechain}
then shows that \(M\) is gradable, and consequently
\(\ell(M)\leq b\).

Indecomposable projective modules are gradable as well. Thus every
indecomposable Gorenstein-projective module is gradable and has
length at most \(b\). Since \(\md\Lambda\) is Krull--Schmidt, every
Gorenstein-projective module is a finite direct sum of
indecomposable Gorenstein-projective modules, and is therefore
gradable.
\end{proof}

We can now prove the main result.

\begin{theorem}\label{thm:main}
Let \(\Lambda\) be a graded Artin algebra which is
Iwanaga--Gorenstein. The following conditions are equivalent:
\begin{enumerate}[label=\textup{(\arabic*)}]
\item \(\Lambda\) is CM-finite.
\item \(\Gproj_\infty(\Lambda)\) has finite type.
\item There is \(b\geq1\) such that \(\ell(M)\leq b\) for every
indecomposable gradable \(M\in\Gproj\Lambda\).
\item \(\GG<\infty\) and \(\Gproj_d(\Lambda)\) has finite type for
every \(d\geq0\).
\end{enumerate}
\end{theorem}

\begin{proof}
\textup{(1)}\(\Rightarrow\)\textup{(2)} is immediate.

\textup{(2)}\(\Rightarrow\)\textup{(4)}.
Let \(b\) be the maximum of the lengths of the finitely many
indecomposable objects of \(\Gproj_\infty(\Lambda)\). If
\(X\in\Gproj^{\Z}\Lambda\) is indecomposable, then \(FX\) is an
indecomposable gradable Gorenstein-projective module, so
\(\ell(FX)\leq b\). By \ref{GG:length},
\(\grl X\leq b\grl\Lambda\). Since \(\grl\Lambda<\infty\), we
obtain \(\GG<\infty\). Moreover,
\(\Gproj_d(\Lambda)\subseteq\Gproj_\infty(\Lambda)\) for every
\(d\geq0\). Hence each \(\Gproj_d(\Lambda)\) has finite type.

\textup{(4)}\(\Rightarrow\)\textup{(3)}.
Put \(G=\GG\), and let \(M\) be an indecomposable gradable
Gorenstein-projective module. By Proposition
\ref{prop:gradedGproj}, there is an indecomposable
\(X\in\Gproj^{\Z}\Lambda\) with \(M\cong FX\). After replacing
\(X\) by a shift, we may assume \(\lb X=0\). Since
\(\grl X\leq G\), one has \(\ub X\leq G-1\). Thus
\(M\in\Gproj_{G-1}(\Lambda)\). By assumption this category has
finite type, so the lengths of its indecomposable objects are
bounded. This proves \textup{(3)}.

\textup{(3)}\(\Rightarrow\)\textup{(1)}.
Let \(b\) be as in \textup{(3)}. For \(r\geq0\), let
\(\cN_r\) be the subset of \(\cI\) consisting of the modules which
occur as the initial term of such a chain of length at most
\(r\) ending at an indecomposable projective module.

We prove by induction on \(r\) that \(\cN_r\) is finite. The set
\(\cN_0\) consists of the representatives in \(\cI\) which are
projective. Since an Artin algebra has only finitely many
indecomposable projective modules, \(\cN_0\) is finite.

Suppose \(\cN_r\) is finite and let \(B\in\cN_{r+1}\). If
\(B\in\cN_r\), there is nothing to consider. Otherwise there is a chain
\begin{equation*}
B=Z_0\longrightarrow Z_1\longrightarrow\cdots
\longrightarrow Z_{r+1}
\end{equation*}
ending at an indecomposable projective module. The tail
\(Z_1\to\cdots\to Z_{r+1}\) is a chain of the form in Definition~\ref{def:componentchain} of length \(r\),
so \(Z_1\in\cN_r\). By Definition~\ref{def:componentchain},
\(B=\widehat E_{Z_1,\ell}\) for some
\(1\leq\ell\leq s_{Z_1}\). Consequently,
\begin{equation*}
\cN_{r+1}\subseteq
\cN_r\cup
\{\widehat E_{C,\ell}:
C\in\cN_r,\ 1\leq\ell\leq s_C\}.
\end{equation*}
The set on the right is finite because \(\cN_r\) is finite and
each \(E_C\) has only finitely many indecomposable summands. Thus
\(\cN_{r+1}\) is finite.

By Proposition~\ref{prop:key}, every indecomposable non-projective
Gorenstein-projective module is isomorphic to the initial term of a chain of components of the chosen
right almost split morphisms of length at most \(2^b-2\) ending at an
indecomposable projective module. Since the initial terms of
these chains belong to \(\cI\), the representative of every
indecomposable non-projective Gorenstein-projective module belongs
to \(\cN_{2^b-2}\). The indecomposable projectives already belong
to \(\cN_0\subseteq\cN_{2^b-2}\). Hence
\(\cI=\cN_{2^b-2}\).

The set \(\cN_{2^b-2}\) is finite. Therefore \(\cI\) is finite,
and \(\Lambda\) is CM-finite.
\end{proof}

\begin{corollary}\label{cor:allgradable}
If \(\Lambda\) is CM-finite, then every Gorenstein-projective
\(\Lambda\)-module is gradable and \(
\underline{\Gproj}\Lambda
\simeq
\underline{\Gproj}^{\Z}\Lambda/\langle 1\rangle\
\).
\end{corollary}

\begin{proof}
The lengths of the indecomposable Gorenstein-projective modules are
bounded, so Proposition~\ref{prop:key} applies.
\end{proof}
We now give a different argument for the gradability of
Gorenstein-projective modules in the CM-finite case. Over an
algebraically closed field, this argument does not require the
Iwanaga--Gorenstein hypothesis and follows directly from a result
of Amiot and Oppermann \cite[Theorem~4.2]{AO}.

\begin{setting}\label{setting}
Let \(k\) be an algebraically closed field and let
\(\Lambda=\bigoplus_{i\in\Z}\Lambda_i\) be a finite-dimensional
graded \(k\)-algebra. No Gorenstein hypothesis is imposed in this
section.
\end{setting}

For \(s\in k^*\), let \(\alpha_s\colon\Lambda\to\Lambda\) be the
algebra automorphism defined by
\(\alpha_s(\lambda)=s^i\lambda\) for \(\lambda\in\Lambda_i\).
For \(M\in\md\Lambda\), denote by \({}^sM\) the corresponding
twist of \(M\).

\begin{lemma}\label{lem:twistfunctor}
For every \(s\in k^*\), the assignment
\(T_s(M)={}^sM\) defines an exact autoequivalence of
\(\md\Lambda\) which preserves indecomposable and
Gorenstein-projective modules.
\end{lemma}

\begin{proof}
Since \(\alpha_s\) is an algebra automorphism, \(T_s\) is an exact
autoequivalence with inverse \(T_{s^{-1}}\). In particular, it
preserves indecomposable modules and projective modules.

If \(P^\bullet\) is a totally acyclic complex of projective
modules, then \(T_sP^\bullet\) is an exact complex of projective
modules. Moreover, for every projective module \(Q\), there is a
natural isomorphism
\begin{equation*}
\Hom_\Lambda(T_sP^\bullet,Q)
\cong
\Hom_\Lambda(P^\bullet,T_{s^{-1}}Q).
\end{equation*}
Since \(T_{s^{-1}}Q\) is projective, the complex on the right is
exact. Thus \(T_sP^\bullet\) is totally acyclic, and \(T_s\)
preserves Gorenstein-projective modules.
\end{proof}

\begin{corollary}\label{thm:twistA}
In Setting~\ref{setting}, if \(\Lambda\) is CM-finite, then every
Gorenstein-projective \(\Lambda\)-module is gradable.
\end{corollary}

\begin{proof}
Let \(G\) be an indecomposable Gorenstein-projective
\(\Lambda\)-module. By Lemma~\ref{lem:twistfunctor},
\({}^sG\) is indecomposable and Gorenstein-projective for every
\(s\in k^*\). Since \(\Lambda\) is CM-finite, the family
\(\{{}^sG:s\in k^*\}\) contains only finitely many isomorphism
classes.

Since \(k\) is algebraically closed, \(k^*\) is infinite. Hence
there is some \(\alpha\in k^*\) such that
\({}^\alpha G\cong{}^\beta G\) for infinitely many
\(\beta\in k^*\). By \cite[Theorem~4.2]{AO}, this is impossible
unless \(G\) is gradable. Thus every indecomposable
Gorenstein-projective module is gradable.

Since \(\md\Lambda\) is Krull--Schmidt and finite direct sums of
gradable modules are gradable, every Gorenstein-projective
\(\Lambda\)-module is gradable.
\end{proof}

\bigskip

\noindent
\textsc{A. M. Bouhada}\\
Department of Mathematics, Champlain College Lennoxville,\\
Williams School of Business, Bishop's University,\\
Sherbrooke, Quebec, Canada\\
\textit{Email address:} \texttt{abouhada@ubishops.ca}

\medskip

\noindent
\textsc{M. Laaraj}\\
Department of Mathematics and Computer Science,\\
Hassan II University of Casablanca,\\
Casablanca, Morocco\\
\textit{Email address:} \texttt{laaraj1-etu@etu-univh2c.ma}

\end{document}